\documentclass[12pt]{article}
\usepackage{amsmath,amsthm,amscd,amsfonts,amssymb,mathrsfs}
\usepackage{latexsym}
\usepackage[usenames]{color}

\usepackage{cite}
\usepackage{titlesec}

\allowdisplaybreaks

\numberwithin{equation}{section}
\titleformat*{\section}{\large\bfseries}

\newcommand{\CC}{\mathbb C}

\newcommand{\RR}{\mathbb R}

\newcommand{\ZZ}{\mathbb Z}

\newtheorem{thm}{Theorem}[section]
\newtheorem{lemma}[thm]{Lemma}
\newtheorem{cor}[thm]{Corollary}
\newtheorem{prop}[thm]{Proposition}

\newtheorem{rem}[thm]{Remark}

\newcommand{\Ex}{\mathop{\mathbb{E}}}
\newcommand{\dprod}[2]{\left\langle #1,#2\right\rangle}

\title{Eigenfunction Statistics for Anderson Model with H\"{o}lder continuous single site Potential}
\author{Dhriti Ranjan Dolai and Anish Mallick\\
The Institute of Mathematical Sciences \\
Taramani, Chennai 600113 \\
 dhriti@imsc.res.in, anishm@imsc.res.in}
\begin{document}

\maketitle
\begin{abstract}
\noindent We consider Random Schr\"{o}dinger operators on $\ell^2(\mathbb{Z}^d)$ with $\alpha$-H\"{o}lder continuous ($0<\alpha\leq 1$) single site distribution. In localized regime we study the distribution of eigenfunctions in space and energy simultaneously. In a certain scaling limit we prove limit points are Poisson.
\end{abstract}
{\bf AMS 2010 MSC:} 35J10, 81Q10, 35P20\\
{\bf Keywords:} Anderson Model, H\"{o}lder continuous measure, Poisson statistics.
\section{Introduction}
\noindent 
The Random Schr\"{o}dinger operators $\{H^{\omega}\}_{\omega\in\Omega}$ on $\ell^2(\mathbb{Z}^d)$ is given by
\begin{equation}
\label{model}
H^\omega=\Delta+V^{\omega},\qquad\omega\in\Omega
\end{equation}
where $\Delta$ is discrete Laplacian defined by
$$
(\Delta u)(n)=\displaystyle\sum_{|m-n|=1}u(m)\qquad\forall~n\in\ZZ^d~u\in\ell^2(\mathbb{Z}^d)
$$
and random potential $V^{\omega}$ is defined by
\begin{equation}
\label{potential}
 V^{\omega}=\displaystyle\sum_{n\in\mathbb{Z}^d}\omega_n|\delta_n\rangle\langle\delta_n|.
\end{equation}
where $\{\delta_n\}_{n\in\mathbb{Z}^d}$ is the standard basis for $\ell^2(\mathbb{Z}^d)$ and $\{\omega_n\}_{n\in\mathbb{Z}^d}$ are real valued iid random variables with 
common probability distribution $\mu$ with compact support. The probability space $(\mathbb{R}^{\mathbb{Z}^d},\mathcal{B}_{\mathbb{R}^{\mathbb{Z}^d}},\otimes_{\mathbb{Z}^d} \mu)$ is constructed via Kolmogorov theorem and will be denoted by $(\Omega,\mathcal{B},\mathbb{P})$, and $\omega_n:\Omega\rightarrow\mathbb{R}$ are projection on $n^{th}$ coordinate.\\
For any bounded set $B\subset\mathbb{R}^d$ we consider the orthogonal projection $\chi_B$ onto $\ell^2(B\cap\mathbb{Z}^d)$ and define the matrices
\begin{equation}
\label{resolvent}
 H^{\omega}_B=\big(\langle\delta_n, H^{\omega}\delta_m\rangle\big)_{n,m\in B},~
G^B(z;n,m)=\langle\delta_{n},(H_B^{\omega}-z)^{-1}\delta_{m}\rangle,~G^B(z)=(H_B^{\omega}-z)^{-1}.
\end{equation}
Note that $H^{\omega}_B$ is the matrix
$$\chi_BH^{\omega}\chi_B : \ell^2(B)\rightarrow \ell^2(B)~ a.e~ \omega.$$ Let $E_{H^{\omega}_B}(\cdot)$ be the spectral projection of $H^{\omega}_B$.\\
Set the resolvent operator and it's matrix elements (Green's function) as:
$$
G(z)=(H^{\omega}-z)^{-1},~~G(z;n,m)=\langle\delta_{n},(H^{\omega}-z)^{-1}\delta_{m}\rangle\qquad z\in\mathbb{C}^{+}.
$$
Throughout this article we will be assuming following two conditions:
\begin{enumerate}
\item[{\bf(a)}] The single site distribution $\mu$ is uniformly $\alpha$-H\"{o}lder continuous for some $0<\alpha\leq 1$.
\item[{\bf(b)}] For any $0<s<1$ there exists $r,C>0$ such that for any $\Lambda\subseteq\ZZ^d$
\begin{equation}
\label{exp}
\sup_{\substack{z\in\CC^{+}\\Re(z)\in[a,b]}} \Ex\left[\left|G^\Lambda(z;n,m)\right|^s\right]\leq C e^{-r|n-m|}
\end{equation}
for any $n,m\in\Lambda$.
\end{enumerate}
When the energy $E$ lies in $[a,b]$, then we say that $E$ is in localized regime. Using the resolvent identity we have
$$\lim_{\Lambda\uparrow\ZZ^d}G^\Lambda(z;n,m)=G(z;n,m)\qquad a.e\ \omega$$
for $z\in\CC^{+}$, so \eqref{exp} holds for $\Ex\left[|G(z;n,m)|^s\right]$ with same constant $C,r$.\\
The condition {\bf (b)} was established by Aizenman-Molchanov \cite{AM} at high disorder for $\alpha$-H\"{o}lder continuous single site distribution. Refer to \cite[inequalities (2.10), (3.19) and (3.20)]{AM} for more details.

\noindent It was shown by Krishna \cite{K}, Combes-Hislop-Klopp \cite{CHK} and Combes-Germinet-Klein \cite{JFA} that whenever the single site distribution is uniformly
$\alpha$-H\"{o}lder continuous the Integrated density of states (IDS) is also uniformly $\alpha$-H\"{o}lder continuous, $0<\alpha\leq1$.\\
Before describing our main result,  we need some notations in place. Let $\nu$ be the Integrated density of states (IDS) for the operator $H^\omega$. Define the fractional derivatives:
\begin{equation}
 \label{frac_deri}
d^{\alpha}_\nu(x)=\lim_{\epsilon\to 0}\frac{\nu(x-\epsilon,x+\epsilon)}{(2\epsilon)^\alpha}~~and~~
D^\alpha_\nu(x)=\varlimsup_{\epsilon\to 0}\frac{\nu(x-\epsilon,x+\epsilon)}{(2\epsilon)^\alpha}
\end{equation}

\noindent Let $g:\mathbb{R}^d\longrightarrow \mathbb{R}^+$ and $f:\mathbb{R}\longrightarrow \mathbb{R}^+$ be compactly supported continuous functions. For a self adjoint operator $H$  on $\ell^2(\mathbb{Z}^d)$ with pure point spectrum i.e $\sigma(H)=\sigma_{pp}(H)$, define $M_g$ as the multiplication operator by $g$:
$$
(M_gu)=g(n)u(n)\qquad\forall n\in\ZZ^d,u\in \ell^2(\mathbb{Z}^d).
$$
Let $\{E_j\}_j$ be the eigenvalues (repeated according to multiplicity) of $H$ and $\psi_j$ be the normalized eigenfunction corresponding to eigenvalue $E_j$. Then
\begin{equation}
 \label{trace_formula}
Tr\big(M_gf(H)\big)=\sum_j\sum_{n\in\mathbb{Z}^d}f(E_j)g(n)|\psi_j(n)|^2.
\end{equation}
Define the random measure $\xi^{\omega}$ on $\mathbb{R}^{1+d}$ by
\begin{equation}
 \label{def1}
\int_{\mathbb{R}\times\mathbb{R}^d} f(E,x)d\xi^{\omega}(E,x)=\sum_j\sum_{n\in\mathbb{Z}^d}f(E_j,n)|\psi_j(n)|^2~~~~\forall~f\in C_c(\mathbb{R}\times\mathbb{R}^d),
\end{equation}
following the notations of \eqref{trace_formula}. Following the notation from physics literature
\begin{equation}\label{PhyNot}
d\xi^\omega(E,x)=\sum_j\sum_n |\psi_j(n)|^2\delta(E-E_j)dE\delta(x-n)dx
\end{equation}
where $\delta(x)$ is the Dirac-delta distribution. So for any Borel set $I\in\mathcal{B}_{\mathbb{R}}$ and $Q\in\mathcal{B}_{\mathbb{R}^d}$ we have
\begin{equation}
 \label{def2}
\xi^{\omega}\big(I\times Q\big)=Tr\big(\chi_{_{Q}}E_{H^{\omega}}(I)\chi_{_{Q}}\big)
\end{equation}
Killip-Nakano \cite{NK} studied eigenfunction statistics for discrete Anderson Model with bounded density. There they studied the sequence of random measures given by
$$\int f(E,x) d\theta_{L,\lambda}^\omega(E,x)=\int_{\mathbb{R}\times \mathbb{R}^d}f\bigg(L^d(E-\lambda),\frac{x}{L}\bigg)d\xi^{\omega}(E,x)~~~\forall~f\in C_c(\mathbb{R}\times\mathbb{R}^d)$$
and proved its convergence to Poisson point process. Nakano in \cite{Nak} worked with Continuum Schr\"{o}dinger operator and was able to show the infinite divisibility of the limiting process. We are interested in similar object, we will study the limit of the \emph{random measures} $\xi^{\omega}_{L,\lambda}$ defined by
\begin{equation}
\label{scale_1}
\int_{\mathbb{R}\times\mathbb{R}^d} f(E,x)d\xi^{\omega}_{L,\lambda}(E,x):
=\int_{\mathbb{R}\times \mathbb{R}^d}f\bigg(\beta_L(E-\lambda),\frac{x}{L}\bigg)d\xi^{\omega}(E,x)~~~\forall~f\in C_c(\mathbb{R}\times\mathbb{R}^d),
\end{equation}
where $\lambda\in[a,b]$ satisfying \eqref{exp}; equivalently
\begin{equation}
\label{scale_2}
\xi^{\omega}_{L,\lambda}\big(I\times Q\big)=Tr\big(\chi_{_{LQ}}E_{H^{\omega}}(\lambda+\beta_L^{-1}I)\chi_{_{LQ}}\big),~~I\in\mathcal{B}_\mathbb{R},
Q\in\mathcal{B}_{\mathbb{R}^d},
\end{equation}
with $\beta_L=L^{d/\alpha}$.  $\beta_L$ is chosen based on the work of Dolai-Krishna \cite{DK1}. They used $\beta_L$ as scaling factor for eigenvalue statistics and showed the convergence to a Poisson random variable.
Our main result is the following theorem.
\begin{thm}
 \label{main}
Let $\xi^\omega_{L,\lambda}$ be defined by \eqref{scale_1} where $H^\omega$ is given by \eqref{model} and $\mu,\lambda$ follow assumptions {\bf (a)} and {\bf (b)}. Let $I\subset\mathbb{R}$ be a bounded symmetric interval and $Q\subset\mathbb{R}^d$ be a rectangles with sides parallel to axes. Then there exists a subsequence
$\{L_n\}$ such that the sequence of random variables $\big\{\xi^{\omega}_{L_n,\lambda}\big(I\times Q\big)\big\}$ converge in
distribution to a Poisson random variable with parameter $|I|^\alpha D^\alpha_\nu(\lambda)|Q|=\gamma_{\lambda}(I\times Q)$, whenever $0< D^\alpha_\nu(\lambda)<\infty$.
\end{thm}
\begin{rem}
 For sequence $\{\xi^{\omega}_{L_n,\lambda}\big(I\times Q\big)\}_n$ to converge, the sequence $\{L_n\}_n$ depends only on $I$ and $\lambda$ but not on $Q$. This is because:
\begin{align*}
\gamma_{\lambda}(I\times Q)&=\lim_{n\rightarrow\infty}\Ex\left[\xi^\omega_{\lambda,L_n}(I\times Q)\right]\\
&=\lim_{n\rightarrow\infty}\Ex\left[\sum_{m\in L Q\cap\ZZ^d}\dprod{\delta_m}{E_{H^\omega}(\lambda+\beta_{L}^{-1}I)\delta_m}\right]\\
&=\lim_{n\rightarrow\infty}(|Q|L_n^d+o(L_n^{d-1}))\Ex[\dprod{\delta_0}{E_{H^\omega}(\lambda+\beta_{L_n}^{-1}I)\delta_0}]\\
&=|Q|\lim_{n\rightarrow\infty} L_n^d\Ex[\dprod{\delta_0}{E_{H^\omega}(\lambda+\beta_{L_n}^{-1}I)\delta_0}]\\
&=|Q|~\lim_{n\rightarrow\infty} L_n^d~\nu(\lambda+\beta_{L_n}^{-1}I)=|Q||I|^\alpha D^\alpha_\nu(\lambda)
\end{align*}
the limit is obtained through lemma \ref{ProcessLimit}.
\end{rem}
\noindent$D^\alpha_\nu$ is defined using symmetric intervals. In general, left and right $\alpha$-derivatives does not coincide with symmetric $\alpha$-derivative, while in case of usual derivative all three are same. It is also hard to determine the set $\{x: D^\alpha_\nu(x)>0\}$, so we have kept $D^\alpha_\nu(\lambda)>0$ in the hypothesis and considered the case of symmetric intervals only.
\\\\
\noindent To compute the limit of $\xi^\omega_{L,\lambda}$ as a random measure over a subsequence $\{L_n\}_n$, we should be able to compute $\lim\limits_{n\rightarrow\infty}\Ex_\omega[\xi^\omega_{L_n,\lambda}(I\times Q)]$ for any bounded interval $I$. Even if we consider $\xi^\omega_{L,\lambda}$ as random measure on the Borel $\sigma$-algebra $\mathscr{B}$ generated by $\{(-b,b)\setminus(-a,a):0<a<b<\infty\}$, we have to take different subsequences for different $I\in \mathscr{B}$. On other hand if $d^\alpha_\nu(\lambda)$ exists, then 
$$\lim_{L\rightarrow\infty}\Ex_\omega[\xi^\omega_{L,\lambda}(I\times Q)]= \alpha 2^{\alpha-1} d^\alpha_\nu(\lambda)|Q|\int_I x^{\alpha-1}dx.$$
where $I$ is a generator of $\mathscr{B}$. In this case one can prove convergence as random measures. As a special case, we can consider $\{\xi^\omega_{L_n,\lambda}(I\times\cdot)\}_n$ as random measure for fixed interval, then

\begin{cor}
\label{main_coro}
For a fixed symmetric bounded interval $I\subset\mathbb{R}$, we consider the random measure $\big\{\xi^{\omega}_{L,\lambda}\big(I\times \cdot\big)\big\}$ on $\mathbb{R}^d$. There exists a subsequence $\{L_n\}$ such that
$\big\{\xi^{\omega}_{L_n,\lambda}\big(I\times \cdot\big)\big\}$ converges weakly to a Poisson point process with intensity measure
$|I|^\alpha D^\alpha_\nu(\lambda)~dx$, where $dx$ is the Lebesgue measure on $\mathbb{R}^d$.
\end{cor}
\noindent Using (iv) of \cite[Theorem 16.16]{kallen}, the above Corollary is immediate once we have Theorem \ref{main}.
\\\\
\noindent Eigenvalue statistics for one dimension was studied by Molchanov \cite{MSA}, and later for higher dimension by Minami \cite{NM}. In region of fractional localization (where \eqref{exp} holds), they showed that the statistics is Poisson. Subsequently the Poisson statistics was shown for the trees by Aizenman-Warzel in \cite{AW} and recently Poisson statistics was obtained by Geisinger \cite{GL} for regular graphs. In recent results Germinet-Klopp \cite{GK} extended the results of \cite{NK}.
\\\\
\noindent Recently Kotani-Nakano \cite{KotNak} investigated the statistics for one dimensional decaying random Schr\"odinger operators on $L^2(\mathbb{R})$. An analogue of Minami's \cite{NM} work was done by Dolai-Krishna \cite{DK1} with $\alpha$-H\"{o}lder continuous single site distribution. In \cite{DK} Dolai-Krishna considered the Anderson Model with decaying Random Potentials and showed that the statistics inside $[-2d, 2d]$ in dimension $d\ge 3$ is independent of the randomness and agrees with that of the free part $\Delta$.
\section{Preliminaries}
Given $L$ large enough, define $l_L$ such that $l_L\approx L^a$ for some $0<a<1$. Define the boxes
\begin{equation}
\label{def_B_P}
 B_p(L)=\{x\in\mathbb{Z}^d: p_jl_L\leq x_j < (p_j+1)l_L,~ for~i=1,2,\cdots d\},~~p\in\mathbb{Z}^d.
\end{equation}
Let $H^{\omega}_{B_p(L)}$ denote the restriction of $H^{\omega}$ to $B_p(L)$. For $\lambda$ in localized regime, define the random measure $\eta^{\omega}_{p,\lambda}$ associated with $H^{\omega}_{B_p(L)}$ by:
\begin{equation}
\label{small_box_1}
\int_{\mathbb{R}\times\mathbb{R}^d}f(E,x)d\eta^{\omega}_{p,\lambda}(E,x)
=\sum_j\sum_{n\in B_p(L)}f\bigg(\beta_L(E_j-\lambda),\frac{n}{L}\bigg)|\psi_j(n)|^2,~f\in C_c(\mathbb{R}\times\mathbb{R}^d),
\end{equation}
where $\{E_j\}_j$ are the eigenvalues of $H^{\omega}_{B_p(L)}$ and $\psi_j$ are corresponding eigenfunctions.
Equivalently
\begin{equation}
 \label{small_box_2}
\eta^{\omega}_{p,\lambda}\big(I\times Q\big)=Tr\big(\chi_{_{LQ}}E_{H^{\omega}_{B_p(L)}}(\lambda+\beta_L^{-1}I)\chi_{_{LQ}}\big),
~~I\in\mathcal{B}_\mathbb{R}, Q\in\mathcal{B}_{\mathbb{R}^d}.
\end{equation}
Since $H^{\omega}_{B_p(L)}$ is a matrix, for $|I|<\infty$ and $|Q|<\infty$ we have,
$$\eta^{\omega}_{p,\lambda}\big(I\times Q\big)<\infty.$$
But it should be noted that it is not a point process.

\noindent Related to $B_p(L)$ we will need:
\begin{center}
 $\partial B_p(L)=\{x\in B_p(L) : \exists ~x'\in \mathbb{Z}^d\setminus B_p(L)~~such ~that~|x-x'|=1\}$\\~\\
$int(B_p(L))=\{x\in B_p(L) : dist (x,\partial B_p(L))>N_L\}$,
\end{center}
where $\{N_L\}_L$ is a increasing sequences of positive integer such that $N_L\approx \gamma \ln L$, we will specify $\gamma$ later. Observe
\begin{equation}
\label{boundary_est}
|B_p(L)\setminus int(B_p(L))|=O(l_L^{d-1}\ln L),~~ N_L\approx \gamma \ln L.
\end{equation}

\noindent Let $C_p(L)$ be the cube in $\mathbb{R}^d$ corresponding to $B_p(L)$ defined by
$$C_p(L)=\{x\in\mathbb{R}^d: p_jl_L\leq x_j < (p_j+1)l_L,~ for~i=1,2,\cdots d\},~~p\in\mathbb{Z}^d.$$
So $B_p(L)=C_p(L)\cap\mathbb{Z}^d$.

\noindent Observe that $\mathbb{Z}^d$ $(resp~\mathbb{R}^d)$ can be expressed as disjoint union of
$B_p(L)$ (respectively $C_p(L)$).\\
For a Borel set $Q$ of finite diameter (i.e $\sup\{|x-y|,x,y\in Q\}<\infty$), there exists a finite finite set $\Gamma$ such that $Q\subseteq \cup_{\Gamma\in p} C_p(L)$. Let $\Gamma_L\subset\mathbb{Z}^d$ be  such that
\begin{equation}
\label{parti}
 LQ=\bigcup_{p\in\Gamma_L}\bigg(C_p(L)\cap LQ\bigg)
\end{equation}
Then
$\{\eta^{\omega}_{p,\lambda}\}_{p\in\Gamma_L}$ are statistically independent. Also
\begin{equation}
 \label{nob}
|\Gamma_L| \leq \bigg(\frac{L}{l_L}\bigg)^d~~|Q|.
\end{equation}
In the following whenever we write sum over $p$, we mean the sum is taken over $\Gamma_L$.

\noindent  We will need Wegner and Minami type estimates given in Combes-Germinet-Klein \cite{JFA}. Hence following there notations, set $S_\mu(s)=\displaystyle\sup_{a\in\mathbb{R}}\mu[a,a+s]$ for probability measure $\mu$ and define
\begin{equation}
\label{def_Q}
Q_\mu(s) = \left\{
 \begin{array}{rl}
  \|\rho\|_\infty~s & ~\text{if} ~\mu~has ~bounded~ density\\
   8S_\mu(s) &~otherwise.\\
 \end{array} \right.
\end{equation}
If $\mu$ is uniformly $\alpha$-H\"{o}lder continuous with $0<\alpha\leq 1$, then $S_\mu(s)\leq U s^\alpha$ for small $s>0$ for some constant $U$.
Following estimates will be used:
\begin{lemma}
 \label{Minami-Wegner}
For all bounded interval $I\subset\mathbb{R}$ and  any finite volume $\Lambda\subset\mathbb{Z}^d$, we have
\begin{equation}
 \label{Wegner}
\mathbb{E}\big(\langle \delta_n, E_{H^{\omega}}(I)\delta_n\rangle\big) \leq Q_{\mu}(|I|)~\forall~n\in\mathbb{Z}^d,
\end{equation}
\begin{equation}
 \label{finite_volume}
\mathbb{E}\big(Tr(E_{H^\omega_\Lambda}(I))\big)\leq Q_{\mu}(|I|)~|\Lambda|,
\end{equation}
\begin{equation}
 \label{Minami}
\mathbb{E}\bigg(Tr(E_{H^\omega_\Lambda}(I))\big(Tr(E_{H^\omega_\Lambda}(I))-1\big)\bigg)\leq \bigg(Q_{\mu}(|I|)~|\Lambda|\bigg)^2.
\end{equation}
\end{lemma}
\noindent Proof can be found in Combes-Germinet-Klein \cite[inequality (2.2)]{JFA} for (\ref{Wegner}), \cite[Theorem 2.3]{JFA} for inequality (\ref{finite_volume}) and \cite[Theorem 2.1]{JFA} for the inequality (\ref{Minami}).\\
The following Corollary is immediate from the above lemma.
\begin{cor}
\label{integrability}
Consider $\nu$ the IDS of the operators $H^\omega$ satisfying the condition $({\bf a})$. Then for  any $\psi \in C_c(\mathbb{R})$ and $n\in\mathbb{Z}^d$ , we have
\begin{equation}
 \int_{\mathbb{R}}\psi(x)d\mathcal{\nu}(x)=\mathbb{E}\big(\langle \delta_n, \psi(H^{\omega})\delta_n\rangle\big)\leq \|\psi\|_{\infty}~Q_\mu(|s_\psi|), ~s_\psi=\mathrm{supp}~\psi.
\end{equation}
\begin{equation}
 \mathbb{E}\big(Tr(\psi(H^{\omega}_\Lambda))\big)\leq \|\psi\|_{\infty}~Q_\mu(|s_\psi|)~|\Lambda|.
\end{equation}
\end{cor}

\begin{prop}
 \label{pro1}
For any $f\in C_c(\mathbb{R}\times\mathbb{R}^d)$ we have,
\begin{equation}
\label{eqiv_con_1}
\mathbb{E}^{\omega}\bigg\{\bigg|\int f(E,x)d\xi^{\omega}_{L,\lambda}(E,x)-\sum_p\int f(E,x)d\eta^{\omega}_{p,\lambda}(E,x)\bigg|\bigg\}
\longrightarrow 0~~as~~L\to\infty.
\end{equation}
\end{prop}
\begin{proof}
Take $f(E,x)=h(E)g(x)$ where $g$ is continuous function with compact support on $\mathbb{R}^d$ and $h$ is of the form
\begin{equation}
\label{resolv_func}
 h(E)=Im\frac{1}{E-z},~Imz>0.
\end{equation}
Since linear combination of functions of form $f$ are dense in $C_c(\mathbb{R}\times \mathbb{R}^d)$, to prove (\ref{eqiv_con_1}) it is sufficient to 
prove for $f$, see \cite[Appendix: The Stone-Weierstrass Gavotte]{HKB} for details. Let supp $g=Q\subset\mathbb{R}^d$ (because $g$ has compact support, we have $\sup\{|x-y|,x,y\in Q\}<\infty$ and $|Q|<\infty$).
We have
\begin{align}
\label{eq1}
\int f(E,x)d\xi^{\omega}_{L,\lambda}(E,x) &=\sum_n g_L(n)\langle \delta_n, h_L(H^{\omega})\delta_n\rangle\\
  &=\frac{1}{\beta_L}\sum_n g_L(n) ImG(z_L;n,n).\nonumber
\end{align}
and
\begin{align}
 \label{eq2}
\sum_p\int f(E,x)d\eta^{\omega}_{p,\lambda}(E,x) &=\sum_p \sum_n g_L(n)\langle \delta_n, h_L(H^{\omega})\delta_n\rangle\\
   &=\frac{1}{\beta_L}\sum_p \sum_n g_L(n) ImG^{B_p}(z_L;n,n),~~B_p=B_p(L)\nonumber
\end{align}
where $g_L(x)=g\big(\frac{x}{L}\big)$, $z_L=\beta_L^{-1}z,~Imz>0$ and $h_L$ is given by
$$h_L(E)=h\big(\beta_L(E-\lambda)\big)=\frac{1}{\beta_L}Im\frac{1}{E-\lambda-\beta_L^{-1}z}.$$
The support of $g_L$ is $LQ$, so from the inequalities (\ref{parti}) and (\ref{nob}) we see that the support of $g_L$ intersect only $O \big(\frac{L}{l_L}\big)^d~(~i.e~|\Gamma_L|)$ many disjoint cubes $B_p(L)$. So from (\ref{eq1}) and (\ref{eq2}) we have 
\begin{align}
 \label{pe4}
\bigg |\int fd\xi^{\omega}_{L,\lambda}-\sum_p\int fd\eta^{\omega}_{p,\lambda}\bigg| 
&=\frac{1}{\beta_L}\bigg|\sum_{n\in LQ}g_L(n) ImG(z_L;n,n)-\sum_{p\in\Gamma_L}\sum_{n\in B_p(L)} g_L(n) ImG^{B_p}(z_L;n,n)\bigg|\\
&\leq \frac{\|g\|_{\infty}}{\beta_L}\sum_{p\in\Gamma_L}\sum_{n\in B_p(L)} \big|ImG(z_L;n,n)-ImG^{B_p}(z_L;n,n)\big|.\nonumber
\end{align}
For $n\in int(B_p(L))$ and $z\in \mathbb{C}^+$, we have the perturbation formula
\begin{equation}
 \label{pertur}
G(z_L;n,n)-G^{B_p}(z_L;n,n)=\sum_{(m,k)\in \partial B_p(L)}G(z_L;n,k)G^{B_p}(z_L;m,n),
\end{equation}
$(m,k)\in \partial B_p(L)$ means $m\in \partial B_p(L)$, $k\in\mathbb{Z}^d\setminus B_p(L)$ such that $|m-k|=1$.
Following steps from Minami \cite{NM}, we use (\ref{pertur}) in (\ref{pe4}) and get
\begin{align}
 \label{differ}
\bigg |\int fd\xi^{\omega}_{L,\lambda}-\sum_p\int fd\eta^{\omega}_{p,\lambda}\bigg|&\leq \frac{\|g\|_{\infty}}{\beta_L}
\sum_{p\in\Gamma_L}\sum_{n\in B_p\setminus int(B_p)}\big[ImG(z_L;n,n)+ImG^{B_p}(z_L;n,n)\big]\\
        &\qquad +\frac{\|g\|_{\infty}}{\beta_L}\sum_{p\in\Gamma_L}\sum_{n\in int(B_p(L))}\sum_{(m,k)\in \partial B_p(L)}
\big|G(z_L;n,k)G^{B_p}(z_L;m,n)\big|\nonumber\\
        & =A_L+B_L\nonumber.
\end{align}
For $B_L$ we have
\begin{align}
 \label{estb}
B_L&=\frac{\|g\|_{\infty}}{\beta_L}\sum_{p\in\Gamma_L}\sum_{n\in int(B_p(L))}\sum_{(m,k)\in \partial B_p(L)}\big|G(z_L;n,k)G^{B_p}(z_L;m,n)\big|\\
&=\frac{\|g\|_{\infty}}{\beta_L}\sum_{p\in\Gamma_L}\sum_{n\in int(B_p(L))}\sum_{(m,k)\in \partial B_p(L)}|G(z_L;n,k)| |G^{B_p}(z_L;m,n)|^s|G^{B_p}(z_L;m,n)|^{1-s}.\nonumber
\end{align}
Now $(m,k)\in \partial B_p(L)$ and $n\in int(B_p(L))$ so we have $|n-k|>N_L$, using the exponential decay of Green's function given in (\ref{exp})
we have 
\begin{equation}
 \label{expdec}
\mathbb{E}^{\omega}(|G^{B_p}(z_L;n,k)|^s)\leq Ce^{-rN_L},
\end{equation}
we also have
$$|G(z_L;n,k)|\leq \frac{1}{|Imz_L|}~~and~~|G^{B_p}(z_L;m,n)|^{1-s}\leq \frac{1}{|Imz_L|^{1-s}}.$$
So using above together with (\ref{expdec}) in (\ref{estb}) we get
\begin{equation}
 \label{estimaB}
\mathbb{E}^{\omega}(B_L) \leq \frac{C~\|g\|_{\infty}}{\beta_L|Imz_L|^{2-s}}|\Gamma_L|~l_L^d ~l_L^{d-1}~N_L~e^{-rN_L}.
\end{equation}
We have $l_L\simeq L^a~(0<a<1)$,~$\Gamma_L=O\big(\frac{L}{l_L}\big)^d$, $Imz_L=\beta_L^{-1}\tau$, $\tau>0$ taking $z=\sigma+i\tau$, and
$\beta_L=L^{d/\alpha}$. Choose $\gamma$ so that
$$\gamma>\frac{1}{r}\left[(1-s)\frac{d}{\alpha}+d+(d-1)a\right]$$
in definition of $N_L$ in \eqref{boundary_est}. Then from (\ref{estimaB}) we get
\begin{equation}
 \label{estiB}
\mathbb{E}^{\omega}(B_L)=O(\gamma L^{-\delta}lnL), ~\delta=r\gamma-\big[(1-s)d/\alpha+d+(d-1)a\big]>0.
\end{equation}
From Combes-Germinet-Klein \cite[A.9]{JFA}  we have, for any $k>0$
\begin{equation}
 \label{diagonal}
Imz~\mathbb{E}[Im G^\Lambda(z;n,n)]\leq \pi\bigg(1+\frac{k}{2}\bigg)S_\mu\bigg(\frac{2~Imz}{k}\bigg).
\end{equation}
Since $Imz_L=\beta_L^{-1}Imz$ with $Im z>0$ so using $S_\mu(s)\leq U s^\alpha$ ($\alpha$-H\"{o}lder continuity of $\mu$) we get 
\begin{align}
 \label{boundary}
\frac{1}{\beta_L}\mathbb{E}\bigg[Im G^{\Lambda}(z_L;n,n)\bigg] &\leq \frac{1}{Imz}~\pi\bigg(1+\frac{k}{2}\bigg)S_\mu\bigg(\frac{2~Imz_L}{k}\bigg),~~\Lambda=C_p,~\Lambda_L\\
&\leq C ~\bigg(\frac{2~\beta_L^{-1}~Imz}{k}\bigg)^{\alpha}\nonumber\\
&\leq C ~L^{-d},~~(\mathrm{since} ~~ \beta_L=L^{d/\alpha})\nonumber.
\end{align}
From (\ref{differ}) and (\ref{boundary_est}) we have
\begin{align}
 \label{estiA}
\mathbb{E}^{\omega}(A_L) &\leq 2 C\frac{\|g\|_{\infty}}{\beta_L}~|\Gamma_L|~|B_p(L)\setminus int B_p(L)|N_L~L^{-d}\\
&\approx C~\bigg(\frac{L}{l_L}\bigg)^d l_L^{d-1}\gamma lnL~L^{-d}\nonumber\\
&=O(L^{-a}lnL),~~l_l=L^a, 0<a<1\nonumber.
\end{align}
Combining (\ref{estiB}) and (\ref{estiA}) gives
$$\mathbb{E}^{\omega}(A_L)+\mathbb{E}^{\omega}(B_L)\xrightarrow{L\rightarrow\infty} 0$$
The above convergence together with (\ref{differ}) completes the proof.\\
\end{proof}

For ease of computation we will define the point process $\tilde{\eta}^\omega_{p,\lambda}$,
\begin{align}
 \label{process}
 \int_{\mathbb{R}\times\mathbb{R}^d}f(E,x)d\tilde{\eta}^{\omega}_{p,\lambda}(E,x)&=\sum_jf\bigg(\beta_L(E_j-\lambda),\frac{pl_L}{L}\bigg),\\
&=Tr\left(f\left(\beta_L(H^\omega_{B_p(L)}-\lambda),\frac{pl_L}{L}\right)\right)\nonumber
\end{align}
where $\{E_j\}_j$ are eigenvalues of $H^\omega_{B_p(L)}$ (following notation of \eqref{PhyNot} we can write $d\tilde{\eta}^\omega_{p,\lambda}(E,x)=\sum_j \delta(\beta_L(E_j-\lambda)-E)dE\delta(x-\frac{p l_L}{L})dx$). One can prove
\begin{equation}
 \label{eqiv_con_2}
 \mathbb{E}^{\omega}\bigg[\bigg |\int fd\xi^{\omega}_{L,\lambda}-\sum_p\int fd\tilde{\eta}^{\omega}_{p,\lambda}\bigg|\bigg] \longrightarrow 0~~as~~L\to\infty
 ,~~f\in C_c(\mathbb{R}\times\mathbb{R}^d)
\end{equation}
by observing that, for $f(E,x)=h(E)g(x)$ where $g\in C_c(\RR^d)$ and $h$ is given by \eqref{resolv_func}, we have
\begin{align}
\label{esti_point}
&\bigg |\int fd\xi^{\omega}_{L,\lambda}-\sum_p\int fd\tilde{\eta}^{\omega}_{p,\lambda}\bigg| \leq \bigg |\int fd\xi^{\omega}_{L,\lambda}-\sum_p\int fd\eta^{\omega}_{p,\lambda}\bigg| +\sum_p\bigg|\int fd\eta^{\omega}_{p,\lambda}-\int fd\tilde{\eta}^{\omega}_{p,\lambda}\bigg|\\
&\qquad\leq \bigg |\int fd\xi^{\omega}_{L,\lambda}-\sum_p\int fd\eta^{\omega}_{p,\lambda}\bigg|\nonumber\\
&\qquad\qquad+\nonumber\sum_p \frac{\max\limits_{n\in B_p(L)}\left|g\left(\frac{n}{L}\right)-g\left(\frac{pl_L}{L}\right)\right|}{\beta_L}\sum_{n\in B_p(L)}\left|Im G(z_L;n,n)-Im G^{B_p}(z_L;n,n)\right|
\end{align}
Repeating the steps of proposition \ref{pro1}, and using density of function of type $f$ in $C_c(\mathbb{R}\times\mathbb{R}^d)$ we have \eqref{eqiv_con_2}. So using $\tilde{\eta}^\omega_{p,\lambda}$ one can note that any limit point of $\xi^\omega_{L,\lambda}$ is a limit point of the point process define by:
\begin{equation}
 \label{etal}
\eta^{\omega}_{L,\lambda}:=\sum_p\tilde{\eta}^{\omega}_{p,\lambda}.
\end{equation}

\begin{rem}
For bounded interval $I\subset\mathbb{R}$ and bounded set $Q\subset\mathbb{R}^d$, using the equation \eqref{eqiv_con_2} on $\chi_I(E)\chi_Q(x)$ we have
\begin{equation}
  \label{characteristic}
\mathbb{E}\big(\big|\xi^{\omega}_{L,\lambda}(I\times Q)-\eta^{\omega}_{L,\lambda}(I\times Q)\big|\big)\rightarrow 0~~as~~L\to\infty.
\end{equation}
\end{rem}

\begin{lemma}
 \label{eqiv_weak_con}
 The weak convergence of $\{\xi^{\omega}_{L,\lambda}\}$ and $\{\eta^{\omega}_{L,\lambda}\}$ are equivalent, i.e
 \begin{equation}
  \label{eqiva}
  \lim_{L\to\infty}\mathbb{E}^{\omega}\bigg[\bigg|e^{-\int fd\xi^{\omega}_{L,\lambda}}-e^{-\int fd\eta^{\omega}_{L,\lambda}}\bigg|\bigg]=0,
  ~~\forall~f\in C_c^{+}(\mathbb{R}\times\mathbb{R}^d).
 \end{equation}
\end{lemma}
\begin{proof}
We have $|e^{-x}-e^{-y}|<|x-y|$ for $x,y>0$; then using this together with (\ref{eqiv_con_2}) will give  (\ref{eqiva}). Hence the lemma.\\
\end{proof}
\begin{lemma}\label{ProcessLimit}
Given $D^\alpha_\nu(\lambda)>0$ and a symmetric bounded interval $I\subset\mathbb{R}$, there exists a sequence $\{L_n\}_n$ such that 
\begin{equation}\label{PrcLimEqn1}
\lim_{n\to\infty}L_n^d\nu(\lambda+\beta_{L_n}^{-1}I)=|I|^\alpha D^\alpha_\nu(\lambda).
\end{equation}
\end{lemma}
\begin{proof}
We have
$$0<D^\alpha_\nu(\lambda)=\displaystyle\varlimsup_{\epsilon\to 0}\frac{\nu(\lambda-\epsilon, \lambda+\epsilon)}{(2\epsilon)^\alpha}<\infty.$$
Choose $\beta^{-1}_{L+1}<\epsilon\leq\beta_L^{-1}$, then for interval $I=[-c,c]$ (for $c>0$) we have
$$\lambda+\epsilon I \subseteq \lambda+\beta_L^{-1}I\ \Rightarrow\ \nu(\lambda+\epsilon I)\leq \nu(\lambda+\beta_L^{-1}I).$$
Using $\beta_{L+1}^\alpha \epsilon^\alpha\ge 1$,
\begin{align}
\frac{\beta_L^\alpha\nu\big(\lambda+\beta_L^{-1}I\big)}{|I|^\alpha} &\ge \bigg(\frac{\beta_L}{\beta_{L+1}}\bigg)^\alpha\frac{\nu\big( \lambda+\epsilon I\big)}
{(\epsilon |I|)^\alpha}~~~\\
&=  \bigg(\frac{\beta_L}{\beta_{L+1}}\bigg)^\alpha\frac{\nu(\lambda-c\epsilon,\lambda+c\epsilon)}{(\epsilon |I|)^\alpha}\nonumber,
\end{align}
From above we get
\begin{align}\label{supeq1}
 \sup_{L\ge M}\frac{\beta_L^\alpha\nu\big(\lambda+\beta_L^{-1}I\big)}{|I|^\alpha} &\ge \bigg(\frac{1}{1+\frac{1}{M}}\bigg)^d
\sup_{\epsilon \in (\beta_{L+1}^{-1}, \beta_L^{-1}],~L\ge M}\frac{\nu\big( \lambda+\epsilon I\big)}{(\epsilon |I|)^\alpha}\nonumber\\
&\ge\bigg(\frac{1}{1+\frac{1}{M}}\bigg)^d \sup_{\epsilon\in(0,\beta_M^{-1}]}\frac{\nu\big( \lambda+\epsilon I\big)}{(\epsilon |I|)^\alpha}.
\end{align}
Here we used the fact that
$$
\bigcup_{L\ge M}(\beta_{L+1}^{-1},\beta_L^{-1}]=(0,\beta_M^{-1}]~~and~~\bigg(\frac{\beta_L}{\beta_{L+1}}\bigg)^\alpha\ge\bigg(\frac{1}{1+\frac{1}{M}}\bigg)^d,~for~L\ge M.
$$
Taking limit $M\to\infty$ in \eqref{supeq1} and using definition of limsup, we get
\begin{equation}
 \label{lower}
\varlimsup_{L\to\infty}\frac{\beta_L^\alpha\nu\big(\lambda+\beta_L^{-1}I\big)}{|I|^\alpha} \ge D^\alpha_\nu(\lambda).
\end{equation}
Similarly starting with  $\epsilon \in (\beta_{L+1}^{-1}, \beta_L^{-1}]$
we get the inequality
$$
\frac{\beta_{L+1}^\alpha\nu\big(E+\beta_{L+1}^{-1}I\big)}{|I|^\alpha} 
\le \bigg(\frac{\beta_{L+1}}{\beta_{L}}\bigg)^\alpha\frac{\nu\big( \lambda+\epsilon I\big)}{(\epsilon |I|)^\alpha}~~~
$$
and proceed as in the above argument, with upper bounds now, to get
\begin{equation}\label{upper}
 \varlimsup_{L\to\infty}\frac{\beta_L^\alpha \nu\big(E+\beta_L^{-1}I\big)}{|I|^\alpha}\le  D^\alpha_{\nu}(\lambda).
\end{equation}
Putting the inequalities (\ref{lower}) and (\ref{upper}) we get
$$
\varlimsup_{L\to\infty}\frac{\beta_L^\alpha \nu\big(\lambda+\beta_L^{-1}I\big)}{|I|^\alpha} =  D^\alpha_{\lambda}(\lambda).
$$
Now using the fact $\beta_L=L^{d/\alpha}$ we have
\begin{equation}
 \label{limsup}
\varlimsup_{L\to\infty}L^d\nu(\lambda+\beta_L^{-1}I)=|I|^\alpha D^\alpha_\nu(\lambda).
\end{equation}
The above imply that there exist a subsequence $\{L_n\}$ such that
$$
\lim_{n\to\infty}L_n^d\nu(\lambda+\beta_{L_n}^{-1}I)=|I|^\alpha D^\alpha_\nu(\lambda).
$$
\end{proof}
\section{ Proof of the Theorem \ref{main}.}
We have
\begin{equation}
 \label{eqi_char_fun}
\mathbb{E}\big|e^{it\xi^{\omega}_{L,\lambda}(I\times Q)}-e^{it\eta^{\omega}_{L,\lambda}(I\times Q)}\big|\leq
 |t|\mathbb{E}\big(\big|\xi^{\omega}_{L,\lambda}(I\times Q)-\eta^{\omega}_{L,\lambda}(I\times Q)\big|\big)
\end{equation}
We  are using  the following fact
\begin{align*}
 |e^{itx}-e^{ity}|^2=2(1-\cos t(x-y))=4\sin^2\frac{t(x-y)}{2}\leq |t(x-y)|^2.
\end{align*}
From (\ref{etal}) and (\ref{parti}) we have
\begin{align}
\label{cal_of_char}
 \mathbb{E}\big[e^{it\eta^{\omega}_{L,\lambda}(I\times Q)}\big]&=\mathbb{E}\big[e^{it\sum_{p\in \Gamma_L}\tilde{\eta}^{\omega}_{p,\lambda}(I\times Q)}\big]\\
&=\mathbb{E}\big[e^{it\tilde{\eta}^{\omega}_{p,\lambda}(I\times Q)}\big]^{|\Gamma_L|}.\nonumber
\end{align}
By definition of $\tilde{\eta}^\omega_{p,\lambda}$ we have
\begin{align}
 \label{expect}
\mathbb{E}\big[e^{it\tilde{\eta}^{\omega}_{p,\lambda}(I\times Q)}\big] &=\sum_{k=0}^\infty e^{itm}\mathbb{P}\big(\tilde{\eta}^{\omega}_{p,\lambda}(I\times Q)=k\big)\\
&=1+\mathbb{E}\big[\tilde{\eta}^{\omega}_{p,\lambda}(I\times Q)\big](e^{it}-1)+R_L\nonumber.
\end{align}
where $R_L$ is given by
\begin{align}
\label{remainder}
R_L &=\sum_{k=0}^\infty e^{itk}\mathbb{P}\big(\tilde{\eta}^{\omega}_{p,\lambda}(I\times Q)=k\big)-1-
\mathbb{E}\big[\tilde{\eta}^{\omega}_{p,\lambda}(I\times Q)\big](e^{it}-1)\\
 &=\sum_{k=0}^\infty e^{itk}\mathbb{P}\big(\tilde{\eta}^{\omega}_{p,\lambda}(I\times Q)=k\big)-
 \sum_{k=0}^\infty\mathbb{P}\big(\tilde{\eta}^{\omega}_{p,\lambda}(I\times Q)=k\big)\nonumber\\
& \qquad-(e^{it}-1)\sum_{k=0}^\infty k~\mathbb{P}\big(\tilde{\eta}^{\omega}_{p,\lambda}(I\times Q)=k\big)\nonumber\\
&=\sum_{k=2}^\infty (e^{itk}-ke^{it}+k-1)\mathbb{P}\big(\tilde{\eta}^{\omega}_{p,\lambda}(I\times Q)=k\big)\nonumber.
\end{align}
Set $I_{L,\lambda}=\lambda+\beta_L^{-1}I$ and using $|e^{itk}-ke^{it}+k-1|\leq 2k$ for $k\ge2$ we get
\begin{align}
 \label{error_estimate}
|R_L| &\leq \sum_{k=2}^\infty |e^{itk}-ke^{it}+k-1|\mathbb{P}\big(\tilde{\eta}^{\omega}_{p,\lambda}(I\times Q)=k\big)\\
&=2\sum_{k=2}^\infty k \mathbb{P}\big(\tilde{\eta}^{\omega}_{p,\lambda}(I\times Q)=k\big)\nonumber\\
&\leq2\sum_{k=2}^\infty k(k-1) \mathbb{P}\big(\tilde{\eta}^{\omega}_{p,\lambda}(I\times Q)=k\big)\nonumber\\
&=2\mathbb{E}\bigg[(\tilde{\eta}^{\omega}_{p,\lambda}(I\times Q)\big(\tilde{\eta}^{\omega}_{p,\lambda}(I\times Q)-1\big)\bigg]\nonumber\\
&\leq 2\mathbb{E}\bigg[(\tilde{\eta}^{\omega}_{p,\lambda}(I\times \mathbb{R}^d)\big(\tilde{\eta}^{\omega}_{p,\lambda}(I\times \mathbb{R}^d)-1\big)\bigg]\nonumber\\
&=2\mathbb{E}\bigg[Tr(E_{H^\omega_{B_p(L)}}(I_{L,\lambda}))\big(Tr(E_{H^\omega_{B_p(L)}}(I_{L,\lambda}))-1\big)\bigg]\nonumber\\
&\leq 2 \big(Q_\mu(|I_{L,\lambda}|)|B_p(L)|\big)^2~~~(using~\ref{Minami})\nonumber\\
&\leq 2 \big(|I_{L,\lambda}|^\alpha~l_L^d\big)^2\nonumber\\
&=O\big(L^{-2d}~l_L^{2d}\big)\nonumber
\end{align}
Using $|\Gamma_L|\simeq O\big(\big(\frac{L}{l_L}\big)^d\big)$ (see (\ref{nob})) in above we get
\begin{align}
\label{error_0}
 |\Gamma_L||R_L| &\leq O \bigg(\frac{l_L^d}{L^d}\bigg)\rightarrow 0~~as~~L\to\infty.
\end{align}
Combining above with (\ref{expect}), (\ref{cal_of_char}) and (\ref{eqi_char_fun}) will give
\begin{align}
\label{same_limit}
 \lim_{L\to\infty}\mathbb{E}\big[e^{it\xi^{\omega}_{L,\lambda}(I\times Q)}]&=\lim_{L\to\infty}\mathbb{E}\big[e^{it\eta^{\omega}_{L,\lambda}(I\times Q)}\big]\\
&=\lim_{L\to\infty}\bigg(1+\frac{|\Gamma_L|\big[\mathbb{E}(\tilde{\eta}^{\omega}_{p,\lambda}(I\times Q))(e^{it}-1)+R_L\big]}{|\Gamma_L|}\bigg)^{|\Gamma_L|}\nonumber\\
&=\lim_{L\to\infty}\bigg(1+\frac{|\Gamma_L|\mathbb{E}(\tilde{\eta}^{\omega}_{p,\lambda}(I\times Q))(e^{it}-1)}{|\Gamma_L|}\bigg)^{|\Gamma_L|}\nonumber.
\end{align}
To compute the limit, we use the subsequence so that \eqref{PrcLimEqn1} holds. Using that subsequence we get:
\begin{align}
 \lim_{n\to\infty}|\Gamma_{L_n}|\mathbb{E}(\tilde{\eta}^{\omega}_{p,\lambda}(I\times Q))
&=\lim_{n\rightarrow\infty}\sum_{p\in\Gamma_{L_n}}\mathbb{E}(\tilde{\eta}^{\omega}_{p,\lambda}(I\times Q))\\
&=\lim_{n\to\infty}\mathbb{E}(\eta^\omega_{{L_n},\lambda}(I\times Q)~~(using~(\ref{etal}))\nonumber\\
&=\lim_{n\to\infty}\mathbb{E}(\xi^\omega_{{L_n},\lambda}(I\times Q)~~(using ~\ref{characteristic})\nonumber\\
&=\lim_{n\to\infty}\sum_{n\in {L_n} Q}\mathbb{E}(\langle \delta_n, E_{H^\omega}(\lambda+\beta_{L_n}^{-1}I)\delta_n\rangle)\nonumber\\
&=|Q|\lim_{{L_n}\to\infty}L_n^d\nu(\lambda+\beta_{L_n}^{-1}I)\nonumber\\
&= |I|^\alpha D^\alpha_\nu(\lambda)|Q|\qquad\ \  (using \eqref{PrcLimEqn1})\nonumber
\end{align}
Using above in the (\ref{same_limit}) together with the fact  $\big(1+\frac{z_n}{n}\big)^n\rightarrow e^z$, 
whenever $z_n\to z$ as $n\to\infty$ gives 
\begin{equation*}
 \mathbb{E}\big[e^{it\xi_{L_n,\lambda}^\omega(I\times Q)}\big]\xrightarrow{n\to\infty}e^{|I|^\alpha D^\alpha_\nu(\lambda)|Q|(e^{it}-1)}.
\end{equation*}
Which shows that $\{\xi_{L_n,\lambda}^\omega(I\times Q)\}$ converges in distribution to a Poisson random variable with parameter $|I|^\alpha D^\alpha_\nu(\lambda)|Q|$.\\\\
{\bf Acknowledgement:} We thank M Krishna for useful discussion and valuable comments. We also thank the referees for their helpful comments and suggestions.

\end{document}